\documentclass{article}
\usepackage{amsmath,amssymb,amsthm}
\usepackage{bm}

\newtheorem{theorem}{Theorem}[section]
\newtheorem{proposition}[theorem]{Proposition}
\newtheorem{lemma}[theorem]{Lemma}
\newtheorem{corollary}[theorem]{Corollary} 
\theoremstyle{definition}
\newtheorem{definition}[theorem]{Definition} 
\theoremstyle{remark}

\newcommand{\N}{\mathbb{N}}
\newcommand{\Z}{\mathbb{Z}}

\newcommand{\C}{\mathbb{C}}

\newcommand{\bmu}{\bm{\mu}}
\newcommand{\bnu}{\bm{\nu}}

\newcommand{\muvec}{\mu_1,\ldots,\mu_p}
\newcommand{\nuvec}{\nu_1,\ldots,\nu_r}

\begin{document}
\title{A generalization of the duality for finite multiple harmonic $q$-series}
\author{Gaku Kawashima}
\date{}
\maketitle

\begin{abstract}
 Recently, Bradley studied partial sums of multiple $q$-zeta values
 and proved a duality result.
 In this paper, we present a generalization of his result.
\end{abstract}

\begin{center}
\end{center}

\section{Introduction}
Recently, finite multiple harmonic sums (MHSs for short) have been
studied in connection with theoretical physics~\cite{BK,V}.
In~\cite{H,Z}, the $p$-divisibility of MHSs for primes $p$ have been 
investigated.
MHSs have a remarkable property known as the duality
and a generalization of this formula, which we call the difference
formula for MHSs, was given in~\cite[Theorem 3.8]{K}.
On the other hand, in~\cite{B1}, 
Bradley proved a $q$-analog of the duality for MHSs.
In the present paper, we shall consider a $q$-analog of the difference
formula for MHSs.
We note that the argument is parallel to that in~\cite{K}.
\par
Here, we explain the duality for finite multiple harmonic $q$-series
due to Bradley.
Let $0 < q < 1$.
The $q$-analog of a non-negative integer $n$ is given by
\begin{displaymath}
 [n]_q = \frac{1-q^n}{1-q}.
\end{displaymath}
For any multi-index (i.e. a finite sequence of positive integers)
$\bmu=(\muvec)$, we define
\begin{align*}
 a_{\bmu}(n) &= \sum_{n=n_1 \ge \cdots \ge n_p \ge 0}
 \frac{q^{(\mu_1-1)(n_1+1)+\cdots+(\mu_p-1)(n_p+1)}}
 {[n_1+1]_q^{\mu_1} \cdots [n_p+1]_q^{\mu_p}}, \quad 0 \le n \in \Z, 
 \intertext{and}
 b_{\bmu}(n) &= \sum_{n=n_1 \ge \cdots \ge n_p \ge 0}
 \frac{q^{(n_2+1)+\cdots+(n_p+1)}}
 {[n_1+1]_q^{\mu_1} \cdots [n_p+1]_q^{\mu_p}}, \quad 0 \le n \in \Z.
\end{align*}
We note that the sum of the infinite series 
$\sum_{n=0}^{\infty}a_{\bmu}(n)$ is the quantity
known as the (non-strict) multiple $q$-zeta value, which has been investigated
in recent years~\cite{B2,B3,OO,OOZ,OT}.
The following is the duality for finite multiple harmonic $q$-series:
\begin{equation}
 \sum_{i=0}^k (-1)^i q^{\frac{i(i+1)}{2}} 
 \begin{bmatrix}
  k \\ i
 \end{bmatrix}_q
 a_{\bmu}(i) = b_{\bmu^{*}}(k), \quad 0 \le k \in \Z, \label{eq1-10}
\end{equation}
where
\begin{displaymath}
 \begin{bmatrix}
  k \\ i
 \end{bmatrix}_q
 = \frac{[k]_q!}{[i]_q!\,[k-i]_q!}
\end{displaymath}
is the $q$-binomial coefficient and $\bmu^{*}$ is the dual multi-index of
$\bmu$.
(The formula is slightly modified from Bradley's one for the purpose of 
generalization.)
The definition of $\bmu^{*}$ will be given in Section \ref{sec3}.
For example, we have
\begin{equation}
 (2,2)^{*} = (1,2,1),\quad (1,1,2)^{*} = (3,1) \quad \text{and} \quad
 (4)^{*} = (1,1,1,1) \label{eq1-100}
\end{equation}
by the diagrams
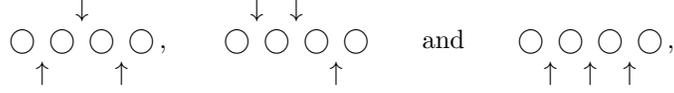
\begin{displaymath}
 {\setlength{\arraycolsep}{0pt}
 \begin{array}{ccccccc}
          &        &        &\downarrow&        &        &         \\
  \bigcirc&        &\bigcirc&          &\bigcirc&        &\bigcirc \\
          &\uparrow&        &          &        &\uparrow&
 \end{array}}\, , \qquad
 {\setlength{\arraycolsep}{0pt}
 \begin{array}{ccccccc}
          &\downarrow&        &\downarrow&        &        &         \\
  \bigcirc&          &\bigcirc&          &\bigcirc&        &\bigcirc \\
          &          &        &          &        &\uparrow&  
 \end{array}} \qquad \text{and} \qquad
 {\setlength{\arraycolsep}{0pt}
 \begin{array}{ccccccc}
          &        &        &        &        &        &         \\
  \bigcirc&        &\bigcirc&        &\bigcirc&        &\bigcirc \\
          &\uparrow&        &\uparrow&        &\uparrow&
 \end{array}}\,,
\end{displaymath}
where the lower arrows are in the complementary slots to the upper
arrows.
\par
We next illustrate the main result of this paper.
For a multi-index $\bmu=(\muvec)$, the quantity
$|\bmu|=\mu_1+\cdots+\mu_p$ is called the weight of $\bmu$.
We introduce nested sums
\begin{multline*}
 c_{\bmu,\bnu}(n,k) = 
 \begin{bmatrix}
  n+k \\ n
 \end{bmatrix}_q^{-1}
 \sum_{\substack{n=n_1 \ge \cdots \ge n_p \ge 0\\ 
 k=k_1 \ge \cdots \ge k_r \ge 0}}
 \frac{q^{(\mu_1-1)(n_1+1)+\cdots+(\mu_p-1)(n_p+1)+k_2+\cdots+k_r}}
 {[n_{i_1}+k_{j_1}+1]_q \cdots [n_{i_m}+k_{j_m}+1]_q}, \\
  0 \le n,\,k \in \Z, 
\end{multline*}
for multi-indices $\bmu=(\muvec)$ and $\bnu=(\nuvec)$ of the same weight
$m$.
The subscripts $i_1,\ldots,i_m$ and $j_1,\ldots,j_m$ are defined by
\begin{align*}
 (i_1,\ldots,i_m) &= 
 (\underbrace{1,\ldots,1}_{\mu_1},\underbrace{2,\ldots,2}_{\mu_2},\ldots,
 \underbrace{p,\ldots,p}_{\mu_p})
 \intertext{and}
 (j_1,\ldots,j_m) &= 
 (\underbrace{1,\ldots,1}_{\nu_1},\underbrace{2,\ldots,2}_{\nu_2},\ldots,
 \underbrace{r,\ldots,r}_{\nu_r}),
\end{align*}
respectively.
For example, for $\bmu=(3,1)$ and $\bnu=(1,1,2)$, we have
\begin{multline*}
 c_{\bmu,\bnu}(n,k) = 
 \begin{bmatrix}
  n+k \\ n
 \end{bmatrix}_q^{-1} \\
 \times
 \sum_{\substack{n=n_1 \ge n_2 \ge 0\\ 
 k=k_1 \ge k_2 \ge k_3 \ge 0}} 
 \frac{q^{2(n_1+1)+k_2+k_3}}
 {[n_{1}+k_{1}+1]_q [n_{1}+k_{2}+1]_q [n_{1}+k_{3}+1]_q
 [n_{2}+k_{3}+1]_q}.
\end{multline*}
The following is the main result of this paper:
For any multi-index $\bmu$, we have
\begin{equation}
 \sum_{i=0}^k (-1)^i q^{\frac{i(i+1)}{2}} 
 \begin{bmatrix}
  k \\ i
 \end{bmatrix}_q
 a_{\bmu}(n+i) = c_{\bmu,\bmu^{*}}(n,k), 
 \quad 0 \le n,\,k \in \Z. \label{eq1-110}
\end{equation}
As we see in Section \ref{sec3}, the equality 
$c_{\bmu,\bmu^{*}}(0,k) = b_{\bmu^{*}}(k)$
holds. Hence the formula (\ref{eq1-110}) is a generalization of the
formula (\ref{eq1-10}).
In Section \ref{sec2}, we interpret the left-hand side of
(\ref{eq1-110}) as the $k$-th $q$-difference of the sequence 
$a_{\bmu} \in \C^{\N}$.
The proof of (\ref{eq1-110}) is given in Section \ref{sec3}.
\section{$q$-differences of a sequence} \label{sec2}
In this section, we define the $k$-th $q$-difference of a sequence 
for a non-negative integer $k$
and give an explicit expression for it.
Throughout this paper, we fix a complex number $q$ equal to neither 
$0$ nor $1$.
(When dealing with multiple $q$-zeta values, we usually assume that
$0 < q < 1$.
But it is not necessary in finite expressions to restrict $q$ to 
the range $0 < q < 1$.)
In the following, we denote by $\N$ the set of non-negative integers.
\begin{definition} 
 \label{def2-10}
 For any $z \in \C$, we define the difference operator 
 $\Delta_z \colon \C^{\N} \to \C^{\N}$ by putting
 \begin{displaymath}
  (\Delta_z a)(n) = a(n) - z a(n+1)
 \end{displaymath}
for any $a \in \C^{\N}$ and any $n \in \N$.
\end{definition}
\begin{definition}
 \label{def2-20}
 For any $k \in \N$, we define the $k$-th $q$-difference operator by
\begin{displaymath}
 \Delta_{q,k} = \Delta_{q^k} \circ \Delta_{q^{k-1}} \circ \cdots \circ
 \Delta_{q},
\end{displaymath}
where $\Delta_{q,0}$ is defined to be the identity on $\C^{\N}$. 
\end{definition}
\begin{definition}
 We define the operator $\nabla_q \colon \C^{\N} \to \C^{\N}$ by putting
 \begin{displaymath}
  (\nabla_q a)(n) = (\Delta_{q,n}a)(0)
 \end{displaymath}
 for any $a \in \C^{\N}$ and any $n \in \N$.
\end{definition}
Let $\C[[X]]$ (resp. $\C[[X,Y]]$) be the ring of formal power series 
in one variable (resp. two variables) over $\C$.
For a sequence $a \in \C^{\N}$, we consider a formal power series
\begin{equation}
 F_{a}(X,Y) 
 = \sum_{n,k=0}^{\infty}(\Delta_{q,k}a)(n) \frac{X^nY^k}{[n]_q!\,[k]_q!}
 \in \C[[X,Y]]. \label{eq2-3}
\end{equation}
The quantities
\begin{displaymath}
 [n]_q = \frac{1-q^n}{1-q} \quad \text{and} \quad 
 [n]_q! = [n]_q [n-1]_{q} \cdots [1]_q
\end{displaymath}
are the $q$-integer and the $q$-factorial, respectively.
As usual, we put $[0]_q! = 1$.
The $q$-derivative of a formal power series $f(X) \in \C[[X]]$ is
defined as
\begin{displaymath}
 D_q f(X) = \left(\frac{d}{dX}\right)_q f(X) =
 \frac{f(qX)-f(X)}{qX-X} \in \C[[X]].
\end{displaymath}
We have the $q$-Leibniz rule
\begin{equation}
 D_q^n (f(X)g(X)) = 
 \sum_{k=0}^n 
 \begin{bmatrix}
  n \\ k
 \end{bmatrix}_q
 (D_q^k f)(X) (D_q^{n-k} g)(q^k X) \label{eq2-5}
\end{equation}
for any $f(X)$, $g(X) \in \C[[X]]$ and any $n \in \N$, where
\begin{displaymath}
 \begin{bmatrix}
  n \\ k
 \end{bmatrix}_q
 = \frac{[n]_q!}{[k]_q!\,[n-k]_q!}
\end{displaymath}
is the $q$-binomial coefficient.
We put $\partial_{X} = (\partial/\partial X)_q$ 
and $\partial_{Y} = (\partial/\partial Y)_q$.
For any $f(X,Y) \in \C[[X,Y]]$, we define
\begin{displaymath}
 \Lambda_{X} f(X,Y) = f(qX,Y)
 \quad \text{and} \quad
 \Lambda_{Y} f(X,Y) = f(X,qY).
\end{displaymath}
By the definition of the $q$-derivative, we have
\begin{equation}
 (1-q) X \partial_X = 1 - \Lambda_X \quad \text{and} \quad
 (1-q) Y \partial_Y = 1 - \Lambda_Y. \label{eq2-7}
\end{equation}
The $q$-commutator of operators $A$ and $B$ is defined as
\begin{displaymath}
 [A,B]_q = AB - qBA.
\end{displaymath}
We have the following $q$-commutation relations:
\begin{alignat}{3}
  &[\partial_X, \Lambda_X]_q &&= [\partial_Y, \Lambda_Y]_q &&= 0, \notag\\
  &[\Lambda_X, X]_q &&= [\Lambda_Y, Y]_q &&= 0, \label{eq2-8} \\ 
  &[\partial_X, X]_q &&= [\partial_Y, Y]_q &&= 1. \notag
\end{alignat}
We note that for a formal power series
\begin{displaymath}
 f(X,Y) =
 \sum_{n,k=0}^{\infty}a(n,k) \frac{X^nY^k}{[n]_q!\,[k]_q!} \in \C[[X,Y]]
\end{displaymath}
the equality
\begin{multline}
 (q\partial_{X}\Lambda_{Y} + \partial_{Y} - 1) f(X,Y) \\
 = \sum_{n,k=0}^{\infty} 
 \left\{q^{k+1} a(n+1,k) + a(n,k+1) - a(n,k)\right\} 
 \frac{X^nY^k}{[n]_q!\,[k]_q!}
 \label{eq2-10}
\end{multline}
holds. From this, we easily see that
\begin{equation}
 (q\partial_{X}\Lambda_{Y} + \partial_{Y} - 1) F_{a}(X,Y) = 0.
 \label{eq2-20}
\end{equation}
\begin{lemma}
 \label{lem2-30}
 If a formal power series $f(X,Y) \in \C[[X,Y]]$ satisfies two
 conditions
 \begin{displaymath}
  (q\partial_{X}\Lambda_{Y} + \partial_{Y} - 1)f(X,Y) = 0
  \quad \text{and} \quad f(X,0)=0,
 \end{displaymath}
 then we have $f(X,Y)=0$.
\end{lemma}
\begin{proof}
 Let
 \begin{displaymath}
  f(X,Y) =
  \sum_{n,k=0}^{\infty}a(n,k) \frac{X^nY^k}{[n]_q!\,[k]_q!} \in \C[[X,Y]]
 \end{displaymath}
 satisfy the two conditions of the lemma.
 Then, by (\ref{eq2-10}), we have 
 \begin{gather*}
  q^{k+1} a(n+1,k) + a(n,k+1) - a(n,k) = 0 \quad \text{for any $n,k \in \N$}
  \intertext{and}
  a(n,0) = 0 \quad \text{for any $n \in \N$}.
 \end{gather*}
 Therefore we obtain the result by using induction on $k$.
\end{proof}
For any sequence $a \in \C^{\N}$, we put
\begin{displaymath}
 f_{a}(X,Y) = 
 \sum_{n=0}^{\infty}a(n) \frac{(X-qY)(X-q^2Y)\cdots(X-q^nY)}{[n]_q!}
 \in \C[[X,Y]].
\end{displaymath}
We note that
\begin{multline}
 \partial_{X} \left\{(X-q^mY)(X-q^{m+1}Y)\cdots(X-q^nY)\right\} \\
 = [n-m+1]_q (X-q^mY)(X-q^{m+1}Y)\cdots(X-q^{n-1}Y) \label{eq2-30}
\end{multline}
and
\begin{multline}
 \partial_{Y} \left\{(X-q^mY)(X-q^{m+1}Y)\cdots(X-q^nY)\right\} \\
 = -q^m[n-m+1]_q (X-q^{m+1}Y)(X-q^{m+2}Y)\cdots(X-q^{n}Y) \label{eq2-40}
\end{multline}
for any integers $1 \le m \le n$,
which are immediate from the definition of the $q$-derivative.
A $q$-analog  of the exponential function is given by
\begin{displaymath}
 e(X) = \sum_{n=0}^{\infty} \frac{X^n}{[n]_q!} \in \C[[X]].
\end{displaymath}
\begin{proposition}
 \label{prop2-40}
 For any sequence $a \in \C^{\N}$, we have
 \begin{displaymath}
  F_{a}(X,Y) = f_{a}(X,Y) e(Y).
 \end{displaymath}
\end{proposition}
\begin{proof}
 It is easily seen that
 \begin{displaymath}
  F_a(X,0) = f_a(X,0) e(0).
 \end{displaymath}
 According to Lemma \ref{lem2-30} and (\ref{eq2-20}), we only have to
 prove the identity
 \begin{equation}
  (q\partial_{X}\Lambda_{Y} + \partial_{Y} - 1)
  \left\{f_{a}(X,Y)e(Y)\right\} = 0. \label{eq2-50}
 \end{equation}
 By (\ref{eq2-30}), (\ref{eq2-40}) and the $q$-Leibniz rule
 (\ref{eq2-5}), we have
 \begin{multline*}
  q \partial_{X} \Lambda_{Y} \left\{f_{a}(X,Y)e(Y)\right\}
  = q \left\{\sum_{n=1}^{\infty}a(n)
  \frac{(X-q^2Y)\cdots(X-q^nY)}{[n-1]_q!}\right\} e(qY)
 \end{multline*}
 and
 \begin{multline*}
  \partial_{Y} \left\{f_{a}(X,Y)e(Y)\right\} \\
  = -q \left\{\sum_{n=1}^{\infty}a(n)
  \frac{(X-q^2Y)\cdots(X-q^nY)}{[n-1]_q!}\right\} e(qY)
  + f_{a}(X,Y) e(Y).
 \end{multline*}
 From these, the identity (\ref{eq2-50}) immediately follows.
\end{proof}
\begin{corollary}
 \label{cor2-50}
 Let $a \in \C^{\N}$ be a sequence. Then, for any $n$, $k \in \N$, we have
 \begin{displaymath}
  (\Delta_{q,k}a)(n) = \sum_{i=0}^{k} (-1)^i q^{\frac{i(i+1)}{2}}
  \begin{bmatrix}
   k \\ i
  \end{bmatrix}_q
  a(n+i).
 \end{displaymath}
\end{corollary}
\begin{proof}
 We apply the operator $\partial_{X}^n \partial_{Y}^k$ to both sides
 of the equation in Proposition \ref{prop2-40}:
 \begin{equation}
  \partial_{X}^n \partial_{Y}^k F_{a}(X,Y) = 
  \partial_{X}^n \partial_{Y}^k \left\{f_{a}(X,Y) e(Y)\right\}. \label{eq2-60}
 \end{equation}
 The right-hand side is equal to
 \begin{displaymath}
  \sum_{i=0}^k 
  \begin{bmatrix}
   k \\ i
  \end{bmatrix}_q
  (\partial_{X}^n \partial_{Y}^i f_{a})(X,Y) e(q^i Y)
 \end{displaymath}
 by the $q$-Leibniz rule (\ref{eq2-5}). Since we have
 \begin{displaymath}
  (\partial_{X}^n \partial_{Y}^i f_{a})(0,0)
  = (-1)^i q^{\frac{i(i+1)}{2}} a(n+i),
 \end{displaymath}
 the desired equality follows from (\ref{eq2-60}) on setting $X=Y=0$.
\end{proof}
\begin{corollary}
 Let $a \in \C^{\N}$ be a sequence. Then, for any $n \in \N$, we have 
 \begin{displaymath}
  (\nabla_{q}a)(n) = \sum_{k=0}^{n} (-1)^k q^{\frac{k(k+1)}{2}}  
  \begin{bmatrix}
   n \\ k
  \end{bmatrix}_q
  a(k).
 \end{displaymath}
\end{corollary}
\begin{proof}
 It follows immediately from Corollary \ref{cor2-50} on setting $n=0$.
\end{proof}
\section{The difference formula for finite multiple harmonic $q$-series}
\label{sec3}
We begin with the definition of the dual of a multi-index.
A multi-index is a finite sequence of positive integers.
For a multi-index $\bmu=(\muvec)$, the quantities
$|\bmu| = \mu_1 + \cdots + \mu_p$ and $l(\bmu) = p$
are called the weight of $\bmu$ and the length of $\bmu$, respectively.
The multi-indices of weight $m$ are in one-to-one
correspondence with the subsets of the set $\{1,2,\ldots,m-1\}$ by
the mapping
\begin{displaymath}
 \mathcal{S}_m \colon (\muvec) \mapsto 
 \{\mu_1,\, \mu_1 + \mu_2,\, \ldots,\, \mu_1 + \mu_2 + \cdots + \mu_{p-1}\}.
\end{displaymath}
For example, in the case $m=3$, we have
\begin{displaymath}
 (3) \mapsto \emptyset, \qquad (1,2) \mapsto \{1\}, \qquad (2,1) \mapsto
 \{2\} \qquad \text{and} \qquad (1,1,1) \mapsto \{1,2\}
\end{displaymath}
from the diagrams
\begin{displaymath}
  {\arraycolsep=1pt
  \begin{array}{ccccc}
          & &        & &         \\
  \bigcirc& &\bigcirc& &\bigcirc \\
          &1&        &2&        
  \end{array}}\,,\qquad
  {\arraycolsep=1pt
  \begin{array}{ccccc}
          &\downarrow&        & &         \\
  \bigcirc&          &\bigcirc& &\bigcirc \\
          &1         &        &2&
  \end{array}}\,, \qquad
  {\arraycolsep=1pt
  \begin{array}{cccccccccc}
          & &        &\downarrow&         \\
  \bigcirc& &\bigcirc&          &\bigcirc \\
          &1&        &2         &
  \end{array}} \qquad \text{and} \qquad
  {\arraycolsep=1pt
  \begin{array}{ccccc}
          &\downarrow&        &\downarrow &         \\
  \bigcirc&          &\bigcirc&           &\bigcirc \\
          &1         &        &2          &
  \end{array}}\,.
\end{displaymath}
\begin{definition}
 \label{def3-10}
 Let $m$ be a positive integer and $\bmu$ a multi-index of weight $m$. 
 Then, we define the dual of $\bmu$ by
 \begin{displaymath}
  \bmu^{*} = \mathcal{S}_m^{-1}(\mathcal{S}_m(\bmu)^c),
 \end{displaymath}
 where $\mathcal{S}_m(\bmu)^c$ denotes the complement of 
 $\mathcal{S}_m(\bmu)$ in the set $\{1,2,\ldots,m-1\}$.
\end{definition}
Examples are given in (\ref{eq1-100}).
We note that the equality
\begin{equation}
 (l(\bmu)-1) + (l(\bmu^{*})-1) = |\bmu| - 1 \label{eq3-10}
\end{equation}
holds for any multi-index $\bmu$.
Now, we state the definition of the finite multiple harmonic $q$-series
which are considered in this paper.
\begin{definition}
 \label{def3-20}
 Let $\bmu=(\muvec)$ be a multi-index. Then, we put
 \begin{align*}
  a_{\bmu}(n) &= \sum_{n=n_1 \ge \cdots \ge n_p \ge 0}
  \frac{q^{(\mu_1-1)(n_1+1)+\cdots+(\mu_p-1)(n_p+1)}}
  {[n_1+1]_q^{\mu_1} \cdots [n_p+1]_q^{\mu_p}}
  \intertext{and}
  b_{\bmu}(n) &= \sum_{n=n_1 \ge \cdots \ge n_p \ge 0}
  \frac{q^{(n_2+1)+\cdots+(n_p+1)}}
  {[n_1+1]_q^{\mu_1} \cdots [n_p+1]_q^{\mu_p}}
 \end{align*}
 for any non-negative integer $n$.
\end{definition}
\begin{definition}
 \label{def3-30}
 Let $\bmu=(\muvec)$ and $\bnu=(\nuvec)$ be multi-indices of the same
 weight $m$. Then, we put
 \begin{displaymath}
  c_{\bmu,\bnu}(n,k) =
  \begin{bmatrix}
   n+k \\ n
  \end{bmatrix}_q^{-1}
  \sum_{\substack{n=n_1 \ge \cdots \ge n_p \ge 0\\ 
  k=k_1 \ge \cdots \ge k_r \ge 0}}
  \frac{q^{(\mu_1-1)(n_1+1)+\cdots+(\mu_p-1)(n_p+1)+k_2+\cdots+k_r}}
  {[n_{i_1}+k_{j_1}+1]_q \cdots [n_{i_m}+k_{j_m}+1]_q}
 \end{displaymath}
 for any non-negative integers $n$ and $k$, where the subscripts
 $i_1,\ldots,i_m,j_1,\ldots,j_m$ are defined by
 \begin{align*}
  (i_1,\ldots,i_m) &= 
  (\underbrace{1,\ldots,1}_{\mu_1},\underbrace{2,\ldots,2}_{\mu_2},\ldots,
  \underbrace{p,\ldots,p}_{\mu_p})
  \intertext{and}
  (j_1,\ldots,j_m) &= 
  (\underbrace{1,\ldots,1}_{\nu_1},\underbrace{2,\ldots,2}_{\nu_2},\ldots,
  \underbrace{r,\ldots,r}_{\nu_r}).
 \end{align*}
\end{definition}
Let $\bmu$ and $\bnu$ be multi-indices of the same weight.
Then, it is easily seen that
\begin{equation}
 c_{\bmu,\bnu}(n,0) = a_{\bmu}(n) \label{eq3-20}
\end{equation}
for any $n \in \N$. Moreover, by (\ref{eq3-10}), we have
\begin{equation}
 c_{\bmu,\bmu^{*}}(0,k) = b_{\bmu^{*}}(k) \label{eq3-30}
\end{equation}
for any $k \in \N$.
For any multi-index $\bmu=(\muvec)$ with $|\bmu| \ge 2$, we define
a multi-index ${}^{-}\!\bmu$ by
\begin{displaymath}
 {}^{-}\!\bmu =
 \begin{cases}
  (\mu_1-1,\mu_2,\ldots,\mu_p) & \text{if $\mu_1 \ge 2$} \\
  (\mu_2,\ldots,\mu_p) & \text{if $\mu_1 = 1$}.
 \end{cases}
\end{displaymath}
We note that
\begin{equation}
 {}^{-}\!(\bmu^{*}) = ({}^{-}\!\bmu)^{*}. \label{eq3-40}
\end{equation}
The following lemma states inductive relations of $c_{\bmu,\bnu}(n,k)$.
\begin{proposition}
 \label{prop3-40}
 Let $\bmu=(\muvec)$, $\bnu=(\nuvec)$ be multi-indices of the same
 weight greater than $1$ and $n$, $k$ non-negative integers. \\
 $\mathrm{(i)}$ If $\mu_1 \ge 2$ and $\nu_1 = 1$, then we have
 \begin{displaymath}
  q^{-n-k-1} \left\{[n+k+1]_q c_{\bmu,\bnu}(n,k) - [k]_q
  c_{\bmu,\bnu}(n,k-1) \right\} = c_{{}^{-}\!\bmu,{}^{-}\!\bnu}(n,k).
 \end{displaymath}
 $\mathrm{(ii)}$ If $\mu_1 = 1$ and $\nu_1 \ge 2$, then we have
 \begin{displaymath}
  [n+k+1]_q c_{\bmu,\bnu}(n,k) - [n]_q c_{\bmu,\bnu}(n-1,k) = 
  c_{{}^{-}\!\bmu,{}^{-}\!\bnu}(n,k).
 \end{displaymath}
\end{proposition}
\begin{proof}
 Since the proof of (ii) is similar to that of (i),
 we prove only (i). We have
 \begin{multline*}
  [n+k+1]_q c_{\bmu,\bnu}(n,k) \\
  =
  \begin{bmatrix}
   n+k \\ n
  \end{bmatrix}_q^{-1}
  \sum_{\substack{n=n_1 \ge \cdots \ge n_p \ge 0\\ 
  k \ge k_2 \ge \cdots \ge k_r \ge 0}}
  \frac{q^{(\mu_1-1)(n_1+1)+\cdots+(\mu_p-1)(n_p+1)+k_2+\cdots+k_r}}
  {[n_{i_2}+k_{j_2}+1]_q \cdots [n_{i_m}+k_{j_m}+1]_q}
 \end{multline*}
 and
 \begin{align*}
  [k]_q &c_{\bmu,\bnu}(n,k-1) \\
  &= [k]_q \frac{[n]_q!\,[k-1]_q!}{[n+k-1]_q!}
  \sum_{\substack{n=n_1 \ge \cdots \ge n_p \ge 0\\ 
  k-1 = k_1 \ge \cdots \ge k_r \ge 0}}
  \frac{q^{(\mu_1-1)(n_1+1)+\cdots+(\mu_p-1)(n_p+1)+k_2+\cdots+k_r}}
  {[n_{i_1}+k_{j_1}+1]_q \cdots [n_{i_m}+k_{j_m}+1]_q} \\
  &=
  \begin{bmatrix}
   n+k \\ n
  \end{bmatrix}_q^{-1}
  \sum_{\substack{n=n_1 \ge \cdots \ge n_p \ge 0\\ 
  k-1 \ge k_2 \ge \cdots \ge k_r \ge 0}}
  \frac{q^{(\mu_1-1)(n_1+1)+\cdots+(\mu_p-1)(n_p+1)+k_2+\cdots+k_r}}
  {[n_{i_2}+k_{j_2}+1]_q \cdots [n_{i_m}+k_{j_m}+1]_q}.
 \end{align*}
 Therefore we obtain
 \begin{multline*}
  [n+k+1]_q c_{\bmu,\bnu}(n,k) - [k]_q c_{\bmu,\bnu}(n,k-1) \\
  =
  \begin{bmatrix}
   n+k \\ n
  \end{bmatrix}_q^{-1}
  \sum_{\substack{n=n_1 \ge \cdots \ge n_p \ge 0\\ 
  k = k_2 \ge \cdots \ge k_r \ge 0}}
  \frac{q^{(\mu_1-1)(n_1+1)+\cdots+(\mu_p-1)(n_p+1)+k_2+\cdots+k_r}}
  {[n_{i_2}+k_{j_2}+1]_q \cdots [n_{i_m}+k_{j_m}+1]_q},
 \end{multline*}
 from which the result follows immediately.
\end{proof}
We restate Proposition \ref{prop3-40} in terms of generating functions.
For multi-indices $\bmu$ and $\bnu$ of the same weight, we define
\begin{displaymath}
 G_{\bmu,\bnu}(X,Y) = \sum_{n,k=0}^{\infty} c_{\bmu,\bnu}(n,k)
 \frac{X^n Y^k}{[n]_q!\, [k]_q!}.
\end{displaymath}
\begin{proposition}
 \label{prop3-50}
 Let $\bmu=(\muvec)$ and $\bnu=(\nuvec)$ be multi-indices of the same
 weight greater than $1$. \\
 $\mathrm{(i)}$ If $\mu_1 \ge 2$ and $\nu_1 = 1$, then we have
 \begin{displaymath}
  q^{-1} \Lambda_{X}^{-1} \Lambda_{Y}^{-1}
  \left(\frac{1-q\Lambda_{X}\Lambda_{Y}}{1-q} - Y\right) 
  G_{\bmu,\bnu}(X,Y) = G_{{}^{-}\!\bmu,{}^{-}\!\bnu}(X,Y).
 \end{displaymath}
 $\mathrm{(ii)}$ If $\mu_1 = 1$ and $\nu_1 \ge 2$, then we have
 \begin{displaymath}
  \left(\frac{1-q\Lambda_{X}\Lambda_{Y}}{1-q} - X\right) 
  G_{\bmu,\bnu}(X,Y) = G_{{}^{-}\!\bmu,{}^{-}\!\bnu}(X,Y).
 \end{displaymath}
\end{proposition}
\begin{proof}
 These are immediate from Proposition \ref{prop3-40}.
\end{proof}
We use Proposition \ref{prop3-50} in order to prove Theorem
\ref{th3-80} by induction, from which the main result follows easily.
We need two lemmas.
\begin{lemma}
 \label{lem3-60}
 $\mathrm{(i)}$ We have
 \begin{multline*}
  (q \partial_X \Lambda_Y + \partial_Y - 1) q^{-1} \Lambda_X^{-1} \Lambda_Y^{-1}
  \left(\frac{1-q\Lambda_X\Lambda_Y}{1-q} - Y\right) \\
  = 
  q^{-2} \Lambda_X^{-1} \Lambda_Y^{-1}
  \left(\frac{1-q^2\Lambda_X\Lambda_Y}{1-q} - qY\right)
  (q \partial_X \Lambda_Y + \partial_Y - 1).
 \end{multline*}
 $\mathrm{(ii)}$ We have
 \begin{displaymath}
  (q \partial_X \Lambda_Y + \partial_Y - 1)
  \left(\frac{1-q\Lambda_X\Lambda_Y}{1-q} - X\right)
  = 
  \left(\frac{1-q^2\Lambda_X\Lambda_Y}{1-q} - X\right)
  (q \partial_X \Lambda_Y + \partial_Y - 1).
 \end{displaymath}
\end{lemma}
\begin{proof}
 (i) By $q$-commutation relations (\ref{eq2-8}), we have
 \begin{align*}
  &[q \partial_X\Lambda_Y + \partial_Y - 1, 
  \frac{1-q\Lambda_X\Lambda_Y}{1-q} - Y]_q \\
  &= [q \partial_X\Lambda_Y + \partial_Y, 
  \frac{1-q\Lambda_X\Lambda_Y}{1-q} - Y]_q
  - (1-q) \left(\frac{1-q\Lambda_X\Lambda_Y}{1-q} - Y\right) \\
  &= (q \partial_X\Lambda_Y + \partial_Y - 1) - (1-q)
  \left(\frac{1-q\Lambda_X\Lambda_Y}{1-q} - Y\right).
 \end{align*}
 We transpose the second term of the right-hand side to the left-hand
 side to obtain
 \begin{align*}
  &(q \partial_X\Lambda_Y + \partial_Y - q) 
  \left(\frac{1-q\Lambda_X\Lambda_Y}{1-q} - Y\right)
  - q \left(\frac{1-q\Lambda_X\Lambda_Y}{1-q} - Y\right)
  (q \partial_X\Lambda_Y + \partial_Y - 1) \\
  &= q \partial_X\Lambda_Y + \partial_Y - 1.
 \end{align*}
 Multiplying by the operator $q^{-2}\Lambda_X^{-1}\Lambda_Y^{-1}$
 from the left, we see that
 \begin{align*}
  &(q \partial_X\Lambda_Y + \partial_Y - 1) q^{-1}\Lambda_X^{-1}\Lambda_Y^{-1}
  \left(\frac{1-q\Lambda_X\Lambda_Y}{1-q} - Y\right) \\
  &\hspace{3cm}- q^{-1}\Lambda_X^{-1}\Lambda_Y^{-1} 
  \left(\frac{1-q\Lambda_X\Lambda_Y}{1-q} - Y\right)
  (q \partial_X\Lambda_Y + \partial_Y - 1) \\
  &= q^{-2}\Lambda_X^{-1}\Lambda_Y^{-1} 
  (q\partial_X\Lambda_Y + \partial_Y - 1),
 \end{align*}
 where we have used the identities
 \begin{displaymath}
  \Lambda_X^{-1}\partial_X = q\partial_X\Lambda_X^{-1} \quad \text{and}
  \quad \Lambda_Y^{-1}\partial_Y = q\partial_Y\Lambda_Y^{-1}.
 \end{displaymath}
 If we transpose the second term of the left-hand side to the
 right-hand side, we obtain the result. \\
 (ii) By a similar computation as in (i), we obtain
 \begin{align*}
  &[q \partial_X\Lambda_Y + \partial_Y - 1, 
  \frac{1-q\Lambda_X\Lambda_Y}{1-q} - X]_q \\
  &= (q \partial_X\Lambda_Y + \partial_Y - 1) - q(1-\Lambda_X)\Lambda_Y
  - (1-q)X(\partial_Y - 1) \\
  &= (1-(1-q)X)(q \partial_X\Lambda_Y + \partial_Y - 1).
 \end{align*}
 The second equality is due to (\ref{eq2-7}).
 From this, the desired identity is easily derived.
\end{proof}
\begin{lemma}
 \label{lem3-70}
 The mappings
 \begin{displaymath}
  q^{-2}\Lambda_X^{-1}\Lambda_Y^{-1}
  \left(\frac{1-q^{2}\Lambda_X\Lambda_Y}{1-q} - qY\right) \quad
  \text{and} \quad 
  \frac{1-q^{2}\Lambda_X\Lambda_Y}{1-q} - X
 \end{displaymath}
 from $\C[[X,Y]]$ to itself are injections.
\end{lemma}
\begin{proof}
 We prove only the first one. The second is similar.
 Since the mapping $q^{-2}\Lambda_X^{-1}\Lambda_Y^{-1}$ is an injection,
 we only have to show that the mapping
 \begin{displaymath}
  \frac{1-q^{2}\Lambda_X\Lambda_Y}{1-q} - qY
 \end{displaymath}
 is an injection. 
 This is obviously a linear mapping.
 We suppose that the formal power series
 \begin{displaymath}
  f(X,Y) = \sum_{n,k=0}^{\infty}a(n,k) \frac{X^nY^k}{[n]_q!\,[k]_q!}
  \in \C[[X,Y]]
 \end{displaymath}
 is in the kernel of the above operator.
 Then we have
 \begin{displaymath}
  [n+k+2]_q a(n,k) - q[k]_q a(n,k-1) = 0
 \end{displaymath}
 for any $n$, $k \in \N$. By induction on $k$, we see that
 $a(n,k)=0$ for any $n$, $k \in \N$. This completes the proof.
\end{proof}
\begin{theorem}
 \label{th3-80}
 For any multi-index $\bmu$, we have
 \begin{displaymath}
  (q\partial_X\Lambda_Y + \partial_Y - 1) G_{\bmu,\bmu^{*}}(X,Y) = 0.
 \end{displaymath}
\end{theorem}
\begin{proof}
 The proof is by induction on $|\bmu|$.
 In the case $|\bmu|=1$ (i.e. $\bmu=(1)$), the theorem follows directly
 from 
 \begin{displaymath}
  G_{(1),(1)}(X,Y) = \sum_{n,k=0}^{\infty} \frac{X^nY^k}{[n+k+1]_q!}.
 \end{displaymath}
 Let $\bmu=(\muvec)$ be a multi-index with $|\bmu| \ge 2$.
 We put $\bmu^{*}=(\mu_1^{*},\ldots,\mu_r^{*})$.
 If $\mu_1 \ge 2$, noting $\mu_1^{*}=1$, we find that
 \begin{displaymath}
  q^{-2}\Lambda_X^{-1}\Lambda_Y^{-1}
  \left(\frac{1-q^2\Lambda_X\Lambda_Y}{1-q} - qY\right)
  (q\partial_X\Lambda_Y + \partial_Y - 1) G_{\bmu,\bmu^{*}}(X,Y) = 0
 \end{displaymath}
 from Lemma \ref{lem3-60} (i), Proposition \ref{prop3-50} (i),
 (\ref{eq3-40}) and
 the hypothesis of induction. According to Lemma \ref{lem3-70}, we have
 \begin{displaymath}
  (q\partial_X\Lambda_Y + \partial_Y - 1) G_{\bmu,\bmu^{*}}(X,Y) = 0.
 \end{displaymath}
 Also in the case $\mu_1 = 1$, we can argue in the same way.
 Therefore we have completed the proof.
\end{proof}
The following is the main result of this paper.
\begin{corollary}
 \label{cor3-90}
 Let $\bmu$ be a multi-index. Then we have
 \begin{displaymath}
  (\Delta_{q,k} a_{\bmu})(n) = c_{\bmu,\bmu^{*}}(n,k)
 \end{displaymath}
 for any $n$, $k \in \N$.
\end{corollary}
\begin{proof}
 By (\ref{eq3-20}), we have $F_{a_{\bmu}}(X,0) = G_{\bmu,\bmu^{*}}(X,0)$.
 (The formal power series $F_a(X,Y)$ is defined in (\ref{eq2-3}) 
 for any sequence $a \in \C^{\N}$.)
 Therefore we obtain 
 \begin{displaymath}
  F_{a_{\bmu}}(X,Y) = G_{\bmu,\bmu^{*}}(X,Y)
 \end{displaymath}
 from Lemma \ref{lem2-30}, (\ref{eq2-20}) and Theorem \ref{th3-80}.
 This implies the corollary.
\end{proof}
As a corollary of Corollary \ref{cor3-90}, we obtain the duality for
finite multiple harmonic $q$-series due to Bradley.
\begin{corollary}
 For any multi-index $\bmu$, we have
 \begin{displaymath}
  \nabla_q a_{\bmu} = b_{\bmu^{*}}.
 \end{displaymath}
\end{corollary}
\begin{proof}
 Since we have (\ref{eq3-30}), the corollary follows 
 from Corollary \ref{cor3-90} on setting $n=0$.
\end{proof}

\begin{flushleft}
 Graduate School of Mathematics \\
 Nagoya University \\
 Chikusa-ku, Nagoya 464-8602, Japan \\
 E-mail: m02009c@math.nagoya-u.ac.jp
\end{flushleft}
\end{document}